\documentclass{amsart}
\usepackage{amssymb,latexsym}
\usepackage{graphicx}
\usepackage[inline]{enumitem}
\usepackage{multicol}
\usepackage{amsmath,amsfonts}  
\usepackage{hyperref}

\graphicspath{ {Images/} }

\newtheorem{theorem}{Theorem}[section]
\newtheorem{prop}[theorem]{Proposition}
\newtheorem{lemma}[theorem]{Lemma}
\newtheorem{cor}[theorem]{Corollary}

\theoremstyle{definition}

\newcommand{\acts}{{\curvearrowright}}

\usepackage{tikz}
\usepackage{verbatim}
\usepackage{color}
\usetikzlibrary{matrix}

\begin{document}
\title{The complexity of the classification problem of continua}
\author{Cheng Chang}
\address{Department of Mathematics\\ University of North Texas\\ 1155 Union Circle \#311430\\  Denton, TX 76203}
\email{chengchang@my.unt.edu}
\author{Su Gao}
\address{Department of Mathematics\\ University of North Texas\\ 1155 Union Circle \#311430\\  Denton, TX 76203}
\email{sgao@unt.edu}

\date{\today}
\thanks{The second author acknowledges the US NSF grant DMS-1201290 for the support of his research.}
\maketitle \thispagestyle{empty}

\begin{abstract} We prove that the homeomorphism problem for connected compact metric spaces is Borel bireducible with a universal orbit equivalence relation induced by a Borel action of a Polish group.
\end{abstract}


\section{Introduction}

By a {\em continuum} we mean a connected compact metric space. To classify all continua up to homeomorphism is generally considered an inaccessible task.  Thus research on this classification problem has been concentrated on subclasses of continua that are in some sense manageable, e.g. homogeneous continua on the plane (c.f. \cite{HOT} \cite{Lewis}). 
However, invariant descriptive set theory (c.f. \cite{G09}) provides a framework in which classification problems in mathematics can be compared to each other, so that one may speak of the relative complexity of classification problems. In this theory each classification problem under consideration is identified as an equivalence relation on a standard Borel space, and central to the theory is the following notion of Borel reducibility which is used to compare the relative complexity of equivalence relations. 

Let $X, Y$ be standard Borel spaces and $E, F$ be equivalence relations on $X, Y$, respectively. We say that $E$ is {\em Borel reducible} to $F$, denoted $E\leq_B F$, if there is a Borel function $\varphi: X\to Y$ such that for all $x, y\in X$, $xEy\iff \varphi(x)F \varphi(y)$. $E$ is said to be {\em Borel bireducible} with $F$, denoted $E\sim_B F$, if both $E\leq_B F$ and $F\leq_B E$.

If $\mathcal{C}$ is a class of equivalence relations and $F\in \mathcal{C}$, we say that $F$ is {\em universal} (or {\em complete}) for $\mathcal{C}$ if for all $E\in \mathcal{C}$, we have $E\leq_B F$. A usual way to characterize the exact complexity of an equivalence relation $E$ is to show that $E$ is Borel bireducible with a universal equivalence relation for a class of equivalence relations, since then it essentially says that $E$ is the most complex object in that class. 

Classification problems in continuum theory have been studied in this descriptive set theoretic setting. In \cite{CDM} the authors considered, among several problems, the homeomorphism problem for dendrites. They showed that this problem, as an equivalence relation, is Borel bireducible with a universal equivalence relation for the class of all orbit equivalence relations induced by a Borel action of the infinite permutation group $S_\infty$.

In general, if $G$ is a Polish group, $X$ is a standard Borel space, and $G\acts X$ is a Borel action, then the {\em orbit equivalence relation} $E^X_G$ is defined by $xE^X_Gy\iff \exists g\in G\ g\cdot x=y$. By a theorem of Becker--Kechris (c.f. \cite{G09} Theorem 3.3.4), for any Polish group $G$, there is a universal equivalence relation in the class of all orbit equivalence relations induced by Borel actions of $G$. It follows from results of Mackey (c.f. \cite{G09} Theorem 3.5.3) and Uspenskij (c.f. \cite{G09} Theorem 2.5.2) that there is a universal equivalence relation in the class of all orbit equivalence relations induced by Borel actions of all Polish groups. We simply refer to such an equivalence relation as a {\em universal orbit equivalence relation}.

The first classification problem determined to be Borel bireducible with a universal orbit equivalence relation is the isometric classification problem for all Polish metric spaces \cite{C} \cite{GK}. Later several other important classification problems in mathematics were also determined to have this exact complexity. These include the isometric classification problem of all separable Banach spaces \cite{M}, the isomorphism problem of all separable (nuclear) $C^*$-algebras \cite{S}, and most recently the homeomorphism problem of all compact metric spaces \cite{Z}. 

The main result of the present paper is the following theorem.

\begin{theorem}\label{main} The homeomorphic classification problem of all continua is Borel bireducible with a universal orbit equivalence relation.
\end{theorem}

The continua we construct will in fact be path-connected. We do not know, however, if our construction can be improved to produce locally connected spaces or homogeneous spaces.

The proof of our main theorem relies on the construction of Zielinski in \cite{Z} for the universality of the homeomorphism problem of all compact metric spaces. We will recall more details of \cite{Z} as we proceed with our proof. Here we just mention that Zielinski's spaces all have Cantor-Bendixson rank 2, and in particular contain infinitely many isolated points.

Recall that the Banach--Stone theorem (c.f. \cite{Se} 7.8.4) implies that for any compact metric spaces $X$ and $Y$, $X$ is homeomorphic to $Y$ if and only if the Banach spaces $C(X,\mathbb{R})$ and $C(Y, \mathbb{R})$ are isometrically isomorphic. Thus Zielinski's theorem \cite{Z} gives an alternative proof of Melleray's theorem on the universality of the isometric problem of separable Banach spaces \cite{M}. 

Furthermore, it is also true that for any compact metric spaces $X$ and $Y$, $X$ is homeomorphic to $Y$ if and only if $C(X, \mathbb{R})$ and $C(Y,\mathbb{R})$ are isometric as Polish metric spaces. Thus it might appear that Zielinski's theorem also gives an alternative proof of the universality of the isometric problem of Polish metric spaces \cite{C} \cite{GK}. However, a close scrutiny unravels that Zielinski's proof depends on Sabok's result on the universality of the affine homeomorphism of Choquet simplices \cite{S}, which in turn relies on the results of \cite{C} and \cite{GK}. Thus this chain of proofs only establishes the logical equivalence of these theorems, and does not give an alternative proof of the earliest result. 

By a theorem of Eilenberg \cite{E}, when $X$ is a compact metric space, $X$ is connected if and only if the Banach space $C(X,\mathbb{R})$ is indecomposable as the direct sum of two subspaces. In contrast, he also showed that $X$ contains an isolated point if and only if $C(X,\mathbb{R})$ admits $\mathbb{R}$ as a direct summand. Thus Zielinski's spaces give rise to Banach spaces with one-dimensional direct summands, whereas spaces constructed from our proof give rise to spaces indecomposable in the sense of Eilenberg. Combining these results, we get the following corollary to our main theorem.

\begin{cor} The isometric classification problem of all separable Banach spaces which are indecomposable as the direct sum of two subspaces is Borel bireducible with a universal orbit equivalence relation.
\end{cor}

We note that this notion of indecomposability is different from the prevailing use of the same terminology in the current literature (where indecomposable often means that when the Banach space is written as a direct sum, one of the summand must be finite-dimensional). We do not know if, for instance, the isometric classification problem of all hereditarily indecomposable separable Banach spaces is still a universal orbit equivalence relation. 

Since our spaces all contain copies of the unit interval, it follows from the Banach--Mazur theorem (c.f. \cite{Se} 8.7.2) that the Banach spaces $C(X,\mathbb{R})$ arising from our proof are all universal, i.e., every real, separable Banach space is isometrically isomorphic to a closed subspace of each of the space $C(X,\mathbb{R})$ we construct.

\begin{cor} The isometric classification problem of all universal separable Banach spaces is Borel bireducible with a universal orbit equivalence relation.
\end{cor}

\section{Preliminaries}

We first explain how to view various classification problems we consider in this paper as equivalence relations on standard Borel spaces, and why they are Borel reducible to orbit equivalence relations induced by Borel actions of Polish groups. Our standard references for notation and terminology are \cite{K} and \cite{G09}. 

Recall that a {\em Polish space} is a separable, completely metrizable topological space. A {\em standard Borel space} is a pair $(X, \mathfrak{B})$, where $X$ is a set and $\mathfrak{B}$ is a $\sigma$-algebra of subsets of $X$, such that $\mathfrak{B}$ is the $\sigma$-algebra generated by some Polish topology on $X$. If $(X, \mathfrak{B})$ is a standard Borel space we refer to elements of $\mathfrak{B}$ as {\em Borel sets}. As usual, if $(X,\mathfrak{B})$ is a standard Borel space and the collection $\mathfrak{B}$ is clear from the context, we will say that $X$ is a standard Borel space. It is natural to view any Polish space as a standard Borel space.

If $X$ and $Y$ are standard Borel spaces, a function $f: X\to Y$ is {\em Borel (measurable)} if for any Borel $B\subseteq Y$, $f^{-1}(B)\subseteq X$ is Borel. 

Given any Polish space $X$, the {\em Effros Borel space} $F(X)$ is the space of all closed subsets of $X$ with the $\sigma$-algebra generated by the all sets of the form
\[
\{F\in F(X): F\cap U\neq \emptyset\},
\]
where $U\subseteq X$ is open. It is a standard Borel space.

A closely related construction is the hyperspace of compact sets. Given any Polish space $X$, we let $K(X)$ denote the space of all compact subsets of $X$, and endow it with the {\em Vietoris topology}, which is the topology generated by sets of the form
$$ \{K\in K(X): K\cap U\neq \emptyset\} $$
and those of the form
$$ \{K\in K(X): K\subseteq U\} $$
where $U\subseteq X$ is open. $K(X)$ is a Polish space.

Throughout this paper we let $\mathcal{Q}$ denote the Hilbert cube $[0,1]^\mathbb{N}$. Since every compact metric space is homeomorphic to a (compact) subspace of $\mathcal{Q}$, we view $K(\mathcal{Q})$ as the Polish space of all compact metric spaces. Let 
$$ \mathcal{C}=\{K\in K(\mathcal{Q})\,:\, \mbox{$K$ is connected}\}. $$
We regard $\mathcal{C}$ as the space of all continua. Note that $\mathcal{C}$ is a closed subspace of $K(\mathcal{Q})$, since for any $K\in K(\mathcal{Q})$, 
$$ K\in\mathcal{C} \iff \forall \mbox{ open } U, V\ (U\cap V=\emptyset\Rightarrow U\cap K=\emptyset \mbox{ or } V\cap K=\emptyset\mbox{ or } K\not\subseteq U\cup V). $$
Thus $\mathcal{C}$ is itself a Polish space.

We only consider real Banach spaces in this paper. The Banach--Mazur theorem (c.f. \cite{Se} 8.7.2) states that $C([0,1],\mathbb{R})$ is a universal separable Banach space. Let
$$ \mathcal{B}=\{F\in F(C([0,1],\mathbb{R}))\,:\, \mbox{$F$ is a linear subspace of $C([0,1],\mathbb{R})$}\}. $$
We regard $\mathcal{B}$ as the space of all separable Banach spaces. It turns out that $\mathcal{B}$ is a Borel subset of $F(C([0,1],\mathbb{R}))$ (c.f. \cite{GJS} for a full discussion), and hence $\mathcal{B}$ is a standard Borel space.

Using the proof of the Banach--Mazur theorem, it can be shown that $X\mapsto C(X,\mathbb{R})$ is a Borel map from $K(\mathcal{Q})$ into $\mathcal{B}$.

Let $H(\mathcal{Q})$ denote the group of all autohomeomorphisms of $\mathcal{Q}$ onto itself, with the composition of maps as the multiplication operation. $H(\mathcal{Q})$ is a Polish group, i.e., a topological group with the underlying topology being Polish. The natural action of $H(\mathcal{Q})$ on $K(\mathcal{Q})$ is continuous. We denote its orbit equivalence relation by $E(\mathcal{Q})$. As noted in \cite{Z}, there is a Borel reduction from $K(\mathcal{Q})$ to $K(\mathcal{Q})$ that reduces the homeomorphism relation to $E(\mathcal{Q})$. This implies, in particular, that the homeomorphism relation on $\mathcal{C}\subseteq K(\mathcal{Q})$ is Borel reducible to an orbit equivalence relation.

By Zielinski's theorem, $E(\mathcal{Q})$ is a universal orbit equivalence relation. 

The isometric classification problem of separable Banach spaces corresponds to the isometry relation on $\mathcal{B}$. This is also known to be Borel reducible to an orbit equivalence relation. Since the discussion of full details of this reduction will need the construction of the universal Urysohn space and its isometry group, which are irrelevant to the rest of this paper, we refer the interested reader to \cite{M} and \cite{FLR} for more details.

For our constructions and proofs of continua we will need the following basic notation and terminology in continuum theory. For unexplained notation and terminology our standard reference is \cite{W}. 

Let $X$ be a continuum. An element $x\in X$ is a {\em cut-point} of $X$ if $X-\{x\}$ is disconnected. If $x$ is not a cut-point of $X$, it is a \emph{noncut-point} of $X$. Cut-points are preserved by homeomorphisms, but not necessarily by continuous maps.

If $X$ is a topological space and $x, y\in X$, a {\em path} from $x$ to $y$ is a continuous function $f: [0,1]\to X$ such that $f(0)=x$ and $f(1)=y$. When there is no danger of confusion, we also refer to the graph of such an $f$ as a path. Define $x\sim y$ iff there is a path from $x$ to $y$, for any $x, y\in X$. Then $\sim$ is an equivalence relation, and its equivalence classes are the {\em path-components} of $X$. $X$ is {\em path-connected} if it has only one path-component, or equivalently, if there is a path from $x$ to $y$ for any $x, y\in X$. 



\section{Coding spaces with isolated points}

Our proof of Theorem~\ref{main} will be built on some constructions by Zielinski \cite{Z} for the universality of the homeomorphism problem of compact metric spaces. We first review some details of \cite{Z} necessary for our proof.

Consider the space whose members are triples of compact metric spaces $(X, B, A)$ such that $X$ is perfect and $A\subseteq B\subseteq X$. Let $\cong_{(1,1)}$ denote the equivalence relation where $(X, B, A)\cong_{(1,1)} (Y, D, C)$ if there exists a homeomorphism $f: X\to Y$ such that $f[A]=C$ and $f[B]=D$. 

\begin{theorem}[Zielinski\cite{Z}] \label{thm:Z}The equivalence relation $\cong_{(1,1)}$ restricted on the space of triples $(X, B, A)$, where $X$ is path-connected and has no cut-points, is Borel bireducible with a universal orbit equivalence relation.
\end{theorem}

In fact, Zielinski showed the much stronger result in which the spaces $X$ in the triples are all the same:
$$ X=\{(x,y)\in\mathcal{Q}^2\,:\, \forall m\neq n\ (y_m=0\mbox{ or } y_n=0)\}. $$
This space is path-connected and has no cut-points.

In the final step of his proof, Zielinski coded the $\cong_{(1,1)}$-type of a triple $(X, B, A)$ by the homeomorphism type of a compact metric space which he denoted by $I_2(X, B, A)$ (we will give the definition below). As we noted before, $I_2(X, B, A)$ has Cantor-Bendixson rank 2, and in particular contains isolated points. In the next section we will modify the construction to get a continuum.

In the rest of this section, we will prove some further results about coding homeomorphism types by spaces with isolated points.

Let $X$ be a compact metric space and $A\subseteq X$ be a closed subspace containing all isolated points of $X$. Let $\mathcal{D}_{X,A}$ be the collection of $D\subseteq X\times (0,1]$ which is a nonempty set of isolated points so that $\overline{D}-D=A\times\{0\}$. If $D\in \mathcal{D}_{X,A}$ and $A\neq\emptyset$, then the set $D$ is necessarily countably infinite. For any $D\in\mathcal{D}_{X,A}$ let $I(X, A; D)=X\times\{0\}\cup D$. Being a closed subspace of $X\times [0,1]$, $I(X, A; D)$ is still a compact metric space.

Zielinski \cite{Z} showed that $\mathcal{D}_{X, A}$ is nonempty for any $X, A$ as above. In fact, let $(q_n)_{n\in\mathbb{N}}$ be any enumeration of a countable dense subset of $A$ with each point repeatedly enumerated infinitely many times, the set $D_A=\{(q_n, 1/n)\,:\, n\geq 1\}$ is an element of $\mathcal{D}_{X, A}$. He also showed that the homeomorphism type of $I(X, A; D_A)$ does not depend on either the choice of the countable dense set or its enumeration. Here we note that $I(X, A; D)$ is unique up to homeomorphism in a slightly broader sense. 

\begin{prop}\label{prop:unique} Let $X$ be a compact metric space and $A\subseteq X$ be a nonempty closed subspace containing all isolated points of $X$. Then for any $D, D'\in\mathcal{D}_{X,A}$, $I(X, A; D)$ and $I(X, A; D')$ are homeomorphic.
\end{prop}

\begin{proof} Fix $D, D'\in \mathcal{D}_{X,A}$. It suffices to define a bijection $g:D\to D'$ such that the map $f: I(X, A; D)\to I(X, A; D')$ defined by 
$$ f(p)=\left\{\begin{array}{ll} p, &\mbox{ if $p\in X\times\{0\}$} \\ g(p), & \mbox{ if $p\in D$} \end{array}\right. $$
is continuous. Then $f$ is a homeomorphism. We define $g:D\rightarrow D'$ by a back-and-forth construction. 

Let $\pi_1: X\times [0,1]\to X$ and $\pi_2:X\times[0,1]\rightarrow [0,1]$ be the projection maps. Let $(d_i)_{i\in\mathbb{N}}$ be a nonrepetitive enumeration of elements of $D$ and $(d_i')_{i\in\mathbb{N}}$ be a nonrepetitive enumeration of elements of $D'$. Fix a compatible metric $\rho$ on $X$. For each $i\in\mathbb{N}$ let $$\delta_i=\rho(\pi_1(d_i), A)=\inf\{\rho(\pi_1(d_i),a)\,:\, a\in A\}$$ and similarly let $\delta_i'=\rho(\pi_1(d_i'),A)$. 

At the beginning of stage $2i$, if $g(d_i)$ has been defined then we do nothing. Suppose $g(d_i)$ has not been defined so far. We define $g(d_i)=d_j'$ where $j\in\mathbb{N}$ is the least such that $g^{-1}(d_j')$ has not been defined, 
$\rho(\pi_1(d_i),\pi_1(d_j'))<\delta_i+2^{-i}$, and $\pi_2(d_j')<2^{-i}$. 
To see that such $d_j'$ exists, fix an $a\in A$ such that $\rho(\pi_1(d_i),a)=\delta_i$. By our assumption $A\times\{0\}=\overline{D'}-D'$, there are infinitely many elements $d\in D'$ such that $\rho(\pi_1(d),a)<2^{-i}$ and $\pi_2(d)<2^{-i}$. They all satisfy $\rho(\pi_1(d_i),\pi_1(d))<\delta_i+2^{-i}$, and not all of them are in the range of $g$ at this time. This finishes the definition of $g$ at stage $2i$.

At stage $2i+1$ we define $g^{-1}(d_i')$ with a similar construction.

Since $g(d_i)$ and $g^{-1}(d_i')$ are defined before stage $2i+2$, it is guaranteed that $g$ is a bijection between $D$ and $D'$.

We verify that $f$ is continuous. The continuity of $f$ at a point $p$ is obvious when $p\in (X-A)\times\{0\}$ or $p\in D$.  Assume $p=(a,0)$ where $a\in A$. Let $U$ be an open set in $I(X, A; D')$ containing $f(p)=(a,0)$. Let $n\in\mathbb{N}$ be such that for all $q\in I(X,A; D')$, $q\in U$ whenever $\rho(\pi_1(q),a)<2^{-n}$ and $\pi_2(q)<2^{-n}$. We define an open set $V$ in $I(X, A; D)$ containing $p$ such that $f[V]\subseteq U$.

Consider the set
$$ S=\{q\in D'\,:\, \rho(\pi_1(q), A)\geq 2^{-n-2} \mbox{ or } \pi_2(q)\geq 2^{-n}\}. $$
$S$ is finite since $\overline{D'}-D'=A\times\{0\}$. Now consider
$$ F=g^{-1}(S)\cup\{d_i, g^{-1}(d_i')\,: \, i\leq n+2\}. $$
Again, $F$ is finite. Let $m\geq n+2$ be large enough such that for any $q\in F$,  either $\rho(\pi_1(q),a)\geq 2^{-m}$ or $\pi_2(q)\geq 2^{-m}$. Define $V$ by
$$ q\in V\iff \rho(\pi_1(q),a)<2^{-m} \mbox{ and } \pi_2(q)<2^{-m}. $$
Then $V$ is open, $V\cap F=\emptyset$, and in particular $g[V]\cap S=\emptyset$. We claim that $f[V]\subseteq U$. For this it suffices to check that for any $d_i\in V\cap D$, $g(d_i)\in U$. Fix $d_i\in V$ and suppose $g(d_i)=d_j'$. We have $i,j>n+2$. Suppose first that $g(d_i)$ was first defined at an even stage, so it is defined at stage $2i$. We have $\delta_i\leq \rho(\pi_1(d_i),a)<2^{-m}\leq 2^{-n-2}$, and so $\rho(\pi_1(d_i),\pi_1(d_j'))<\delta_i+2^{-i}< 2^{-n-1}$, from which it follows that $\rho(\pi_1(d_j'), a)<2^{-m}+2^{-n-1}<2^{-n}$. Also $\pi_2(d_j')<2^{-i}<2^{-n}$ by the definition of $d_j'=g(d_i)$. Thus $d_j'\in U$ as required.

Finally suppose that $g(d_i)$ was first defined at an odd stage, that is, $j<i$ and $g^{-1}(d_j')=d_i$ was defined at stage $2j+1$. Still note that $j>n+2$ by the definition of $F$ and $V$. Note that $d_j'\not\in S$, so $\delta_j'=\rho(\pi_1(d_j'),A)<2^{-n-2}$. From our construction, $\rho(\pi_1(d_j'), \pi_1(d_i))<\delta_j'+2^{-j}<2^{-n-1}$. Since $\rho(\pi_1(d_i),a)<2^{-m}<2^{-n-1}$, we have that
$\rho(\pi_1(d_j'), a)<2^{-n}$. From $d_j'\not\in S$ again we get $\pi_2(d_j')<2^{-n}$. This again shows that $d_j'\in U$ as required.
\end{proof}

In view of Proposition~\ref{prop:unique}, we simply write $I(X,A)$ for any $I(X,A; D)$ for $D\in \mathcal{D}_{X,A}$. If $A$ is empty, we let $I(X, A)=I(X, A; D)$ where $D$ is a singleton. Then the proposition will formally hold.

From Proposition 1 of \cite{Z}, it now follows easily that $I(X,A)$ is a coding space for the homeomorphism type of pairs $(X, A)$ where $X$ is a compact metric space and $A\subseteq X$ is a closed subspace. We state this result without proof.

\begin{prop} Let $X, Y$ be compact metric spaces, and $A\subseteq X$ and $B\subseteq Y$ be closed subspaces containing all isolated points of $X$ and $Y$, respectively. Then the following are equivalent:
\begin{enumerate}
\item[\rm (i)] $(X, A)\cong (Y, B)$, i.e., there is a homeomorphism $f: X\to Y$ with $f[A]=B$.
\item[\rm (ii)] $I(X, A)$ and $I(Y, B)$ are homeomorphic.
\end{enumerate}
\end{prop}

The map $i: X\to I(X, A)$ defined by $i(x)=(x,0)$ is a natural homeomorphic embedding of $X$ into $I(X, A)$. Under this embedding, $X$ can be viewed as a closed subspace of $I(X, A)$. In fact, $X$ is now the subspace of all limit points in $I(X, A)$, i.e., the Cantor-Bendixson derivative of $I(X, A)$. Thus the Cantor-Bendixson rank of $I(X, A)$ is exactly $1+\alpha$, where $\alpha$ is the Cantor-Bendixson rank of $X$.

Continue to assume $X$ is a compact metric space. For any increasing sequence of closed subspaces $A_1\subseteq A_2\subseteq \cdots \subseteq A_n$, where $A_1$ contains all isolated points of $X$, define $I_n(X, A_1, \dots, A_n)$ inductively as follows. Let $I_1(X, A_1)=I(X, A_1)$ and let $D_1=I_1(X, A_1)-X$. In general assume $I_k(X, A_1,\dots, A_k)$ has been defined, and let $D_k=I_k(X, A_1, \dots, A_k)-I_{k-1}(X, A_1, \dots, A_{k-1})$. Each $D_k$ is the set of all isolated points in $I_k(X, A_1, \dots, A_k)$. Define $$I_{k+1}(X, A_1, \dots, A_{k+1})=I(I_k(X, A_1, \dots, A_k), D_k\cup A_{k+1}). $$
As in the previous theorem, denote $(X, \vec{A})\cong (Y, \vec{B})$ if there is a homeomorphism $f: X\to Y$ such that $f[A_i]=B_i$ for all $i$. It follows from Proposition~\ref{prop:unique} that the homeomorphism type of $I_n(X, \vec{A})$ does not depend on exactly how the extra points of $I_n(X, \vec{A})-X$ are selected. Note that we do not assume that $A_k$ are nonempty. In general, if $A_k$ is empty, then each $A_i$, $i<k$, is also empty, and it follows that $X$ has no isolated points. In this case $I_k(X, A_1, \dots, A_k)$ is homeomorphic to a disjoint union of $X$ with a countable compact space of Cantor-Bendixson rank $k$. We have a generalization of the above proposition as follows.

\begin{prop} Let $X, Y$ be compact metric spaces, and $\vec{A}$ and $\vec{B}$ be increasing sequences of $n$ many closed subspaces of  $X$ and $Y$ respectively, with $A_1$ and $B_1$ contain all isolated points of $X$ and $Y$ respectively. Then the following are equivalent:
\begin{enumerate}
\item[\rm (i)] $(X, \vec{A})\cong (Y, \vec{B})$.
\item[\rm (ii)] $I_n(X, \vec{A})$ and $I_n(Y, \vec{B})$ are homeomorphic.
\end{enumerate}
\end{prop}

This is only a formal generalizaton of Zielinski's construction of $I_2(X, B, A)$. Note that our notation is slightly different from his.

In general, if $S_1, \dots, S_n \subseteq X$ are closed subspaces of a compact metric space $X$, we code the homeomorphism type of $(X, \vec{S})$. Assume that each $S_k$ contains all isolated points of $X$. We define an increasing sequence $A_1\subseteq \cdots\subseteq A_{2^n-1}$ by induction as follows. 

If $n=1$ let $A_1=S_1$ and we are done. In general, if an increasing sequence $B_1\subseteq \cdots\subseteq B_{2^k-1}$ has been defined for $S_1,\dots, S_k$, we define a sequence $A_1\subseteq\cdots \subseteq A_{2^{k+1}-1}$ for $S_1, \dots, S_k, S_{k+1}$ by
$$
A_i=\left\{\begin{array}{ll} B_i\cap S_{k+1}, & \mbox{ if $1\leq i\leq 2^k-1$} \\
S_{k+1}, & \mbox{ if $i=2^k$} \\ B_{i-2^k}\cup S_{k+1}, & \mbox{ if $2^k+1\leq i\leq 2^{k+1}-1$} \end{array}\right.
$$
This sequence allows all $2^k-1$ many elements of the Boolean algebra of sets generated by $S_1, \dots, S_n$ to be recovered. We denote the sequence as $\vec{A}(\vec{S})$. It is easy to check that $(X, \vec{S})\cong (Y, \vec{T})$ if and only if $(X, \vec{A}(\vec{S}))\cong (Y, \vec{A}(\vec{T}))$.

Combined with the above coding, the homeomorphism type of $(X, \vec{S})$ can be coded by the homeomorphism type of a single compact metric space, which we denote by $I(X, \vec{S})$. If $X$ has Cantor-Bendixson rank $\alpha$ and $\vec{S}$ has length $k$, then the Cantor-Bendixson rank of $I(X, \vec{S})$ is at most $2^k-1+\alpha$. 

\begin{prop}\label{prop:gen} Let $X, Y$ be compact metric spaces, and $\vec{S}$ and $\vec{T}$ be sequences of closed subspaces of  $X$ and $Y$ respectively, with each $S_i$ and each $T_i$ contain all isolated points of $X$ and $Y$ respectively. Then the following are equivalent:
\begin{enumerate}
\item[\rm (i)] $(X, \vec{S})\cong (Y, \vec{T})$.
\item[\rm (ii)] $I(X, \vec{S})$ and $I(Y, \vec{T})$ are homeomorphic.
\end{enumerate}
\end{prop}

\section{Connected coding spaces}

Let $X$ be a compact metric space. We define the {\em fan space} $F_X$ of $X$ as the quotient of $X\times[0,1]$ by the equivalence relation $\sim$ defined as
$$ (x, s)\sim (y, t)\iff (x, s)=(y, t) \mbox{ or } s=t=1. $$
The point $[(x,1)]_\sim$ in $F_X$ is a distinguished point; we denote it by $\gamma_X$ and call it the {\em apex}. $X$ can be viewed, again in a canonical way, as a subspace of $F_X$.

$F_X$ is obviously compact. We note that it can be given a canonical metric:
$$ d_F((x,s), (y,t))=2|s-t|+(1-\max\{s,t\})\rho(x,y), $$
where $\rho<1$ is a compatible metric on $X$. $F_X$ is also clearly a path-connected space: for every point $(x,s)$ there is a canonical path $P$ from $(x, s)$ to $\gamma_X$, namely, 
$$ P(\tau)=(x, s+\tau(1-s)) \mbox{ for $\tau\in[0,1]$}. $$
Therefore $F_X$ is a path-connected continuum. 

\begin{lemma}\label{lem:1} Let $X, Y$ be compact metric spaces. Any homeomorphism $f: X\to Y$ can be extended to a homeomorphism $\tilde{f}: F_X\to F_Y$.
\end{lemma}

\begin{proof} Just define $\tilde{f}(x, s)=(f(x), s)$ for $x\in X$ and $s\in[0,1]$. \end{proof}

Next we code pairs $(X, A)$. Given a compact metric space $X$ and a closed subspace $A\subseteq X$, define $F(X, A)$ as a subspace of the fan space $F_X$:
$$ F(X, A)=\{ [(x,s)]_\sim\in F_X\,:\, s=0 \mbox{ or } x\in A\}. $$
Alternatively, we consider the equivalence relation $\sim$ defined above, restricted to the space 
$$ X\times\{0\}\cup A\times [0,1]. $$
$F(X,A)$ is the again the quotient space given by $\sim$.

There is a canonical homeomorphic copy of $X$ in $F(X, A)$, namely $X\times\{0\}$, and a canonical homeomorphic copy of $F_A$ in $F(X, A)$. We will denote the distinguished point $\gamma_{F_A}$ by $x^*$, when there is no danger of confusion.

It is easy to see that 
if $X$ is (path-)connected, then so is $F(X,A)$. 

\begin{lemma}\label{lem:2} Let $X, Y$ be compact metric spaces and $A\subseteq X$ and $B\subseteq Y$ be closed subspaces. For any homeomorphism $f: X\to Y$ with $f[A]=B$ there is an extended homeomorphism $\tilde{f}: F(X, A)\to F(Y, B)$. Thus if $(X, A)\cong (Y, B)$ then $F(X, A)$ is homeomorphic to $F(Y, B)$.
\end{lemma}

\begin{proof} Lemma~\ref{lem:1} gives an extended homeomorphism $\tilde{f}:F_X\to F_Y$, which can be restricted to $F(X,A)$ to give the required homeomorphism to $F(Y, B)$.
\end{proof}

The next coding space $J(X,A)$ is based on the space $I(X, A)$ defined in the previous section. Write $I(X, A)=X\cup D$, where $D$ is the set of all isolated points in $I(X, A)$. Note that $\overline{D}=D\cup A$. We define $$J(X, A)=F(I(X, A), \overline{D}).$$
In general, if $\vec{S}$ is a finite sequence of closed subspaces of $X$ such that each $S_i$ contains all isolated points of $X$, we write $I(X,\vec{S})=X\cup D$ as a disjoint union (note that $D$ is no longer necessarily the set of all isolated points). Note that $\overline{D}=D\cup\bigcup \vec{S}$. We define
$$ J(X,\vec{S})=F(I(X,\vec{S}), \overline{D}). $$

The following is our main technical theorem on connected coding spaces.

\begin{theorem}\label{thm:J} Let $X, Y$ be continua without cut-points and $\vec{S}, \vec{T}$ be sequences of closed subspaces of $X, Y$ respectively. Then the following are equivalent:
\begin{enumerate}
\item[\rm (i)] $(X, \vec{S})\cong (Y, \vec{T})$.
\item[\rm (ii)] $J(X, \vec{S})$ and $J(Y, \vec{T})$ are homeomorphic.
\end{enumerate}
\end{theorem} 

\begin{proof} Since $X$ and $Y$ are connected, they do not contain any isolated points. 

We first show (i)$\Rightarrow$(ii). Suppose $(X,\vec{S})\cong (Y, \vec{T})$. From Proposition~\ref{prop:gen} there is a homeomorphism $h$ from $I(X, \vec{S})$ to $I(Y, \vec{T})$. If the length of $\vec{S}$ is $k$, then there is $1\leq n\leq 2^k-1$ such that $X$ is the $n$-th Cantor-Bendixson derivative of $I(X,\vec{S})$. In this case the length of $\vec{T}$ is also $k$, and $Y$ is also the $n$-th Cantor-Bendixson derivative of $I(Y,\vec{T})$.  Writing $I(X, \vec{S})=X\cup D_X$ and $I(Y, \vec{T})=Y\cup D_Y$ as disjoint unions, we have $h[X]=Y$ and $h[\overline{D}_X]=\overline{D}_Y$. From Lemma~\ref{lem:2}, we have that $J(X,\vec{S})$ and $J(Y,\vec{T})$ are homeomorphic.

In the rest of the proof we show (ii)$\Rightarrow$(i). For this, it suffices to consider the special case where $\vec{S}$ and $\vec{T}$ are increasing sequences of closed subspaces of $X$ and $Y$ respectively. Our proof will proceed by induction on the length of $\vec{S}$. 

We first fix some notation. Suppose $S_1\subseteq \dots\subseteq S_k$. Now 
$$ I(X, \vec{S})=I_k(X, S_1, \dots, S_k). $$
We write $I(X,\vec{S})$ as a disjoint union $X\cup D_X$. Then 
$$ J(X,\vec{S})=F(I(X,\vec{S}), \overline{D}_X). $$
By slightly abusing the notation, we think of $J(X,\vec{S})$ as a subset of $I(X,\vec{S})\times[0,1]$ (where in fact it is a subset of a quotient). Thus each element of $J(X,\vec{S})$ is a pair $(u, s)$, where $u\in I(X, \vec{S})$ and $s\in [0,1]$. The distinguished point of $J(X,\vec{S})$, that is, the unique point with $s=1$, is denoted as $x^*$. If $z_0=(u_0,s_0)$ where $u_0\in \overline{D}_X$ and $s_0\in [0,1)$, then $\{(u_0,s)\,:\, s\in[s_0,1]\}$ is a path in $J(X,\vec{S})$ from $z_0$ to $x^*$; we refer to this path as the {\em canonical path} from $z_0$ to $x^*$.

Let $E_X$ be the set of all isolated points in $I(X,\vec{S})$. 
Easily $E_X\subseteq D_X$. By an easy induction one can show that $\overline{E}_X=\overline{D}_X$, that is, $E_X$ is dense in $D_X$. Since $S_1\subseteq\dots\subseteq S_k$, we also have $\overline{E}_X\cap X=S_k$.

We will use similar notation $D_Y, E_Y, y^*$ and similar convention on the $Y$ side.

For the record we note that $J(X,\vec{S})$ is not always connected. Consider the case $S_k=\emptyset$. It follows that $S_1=\dots=S_k=\emptyset$. Now $D_X=\overline{D}_X$ is itself a countable compact metric space of Cantor-Bendixson rank $k$. In this case $J(X, \vec{S})$ has exactly two components, $X\times\{0\}$ and $F_{D_X}$. This is the only case when $J(X,\vec{S})$ is not connected. In fact, if $S_k\neq\emptyset$, then $\overline{D}_X\cap X=S_k\neq\emptyset$, and $J(X, \vec{S})$ is now connected because it is a non-disjoint union of two connected subspaces $X\times\{0\}$ and $F_{D_X}$. 

We argue that the theorem is still true in the case $S_k=\emptyset$. To continue the analysis of the situation, we note that $F_{D_X}$ always has a cut-point. In fact, $D_X$ contains at least one isolated point $u_0$, and therefore every point on the canonical path from $z_0=(u_0,0)$ to $x^*$ is a cut-point. Since $X$ has no cut-points, neither does $X\times\{0\}$, and therefore the two components $X\times\{0\}$ and $F_{\overline{D}_X}$ are not homeomorphic. Similar analysis applies to $J(Y, \vec{T})$. Now suppose $S_k=\emptyset$ and $J(X,\vec{S})$ is homeomorphic to $J(Y, \vec{T})$. It follows that each $T_i=\emptyset$. Moreover, if $h$ is a homeomorphism from $J(X, \vec{S})$ to $J(Y, \vec{T})$, then $h[X\times\{0\}]=Y\times\{0\}$ and $h[F_{D_X}]=F_{D_Y}$ since any homeomorphism must send a component to a component. In particular $X$ is homeomorphic to $Y$, and the length of $\vec{S}$ equals the length of $\vec{T}$. Since any homeomorphism sends $\emptyset$ to $\emptyset$, we have $(X,\vec{S})\cong (Y, \vec{T})$.

For the rest of the proof we assume $S_k\neq\emptyset$. 

First note that $x^*$ is a cut-point of $J(X,\vec{S})$ so that $J(X,\vec{S})-\{x^*\}$ has infinitely many components. In fact, for each $u\in E_X$, the canonical path in $J(X,\vec{S})$ from $(u, 0)$ to $x^*$, less $x^*$, is a component of $J(X,\vec{S})-\{x^*\}$. Since $E_X$ is infinite, $J(X,\vec{S})-\{x^*\}$ has infinitely many components. We argue that $x^*$ is the only point in $J(X, \vec{S})$ with this property. 

Let $z_0\in J(X, \vec{S})$. Assume $z_0=(u_0, s_0)$ where $u_0\in I(X,\vec{S})$ and $s_0\in [0,1)$. We claim that $J(X,\vec{S})-\{z_0\}$ has at most two components. Consider two cases: (A) $u_0\in D_X$ and (B) $u_0\not\in D_X$. 

First consider case (A): $u_0\in D_X$. Let 
$$ L=\{ (u_0, s)\,:\, s\in [0,s_0)\} \mbox{ (if $s_0=0$ then $L=\emptyset$)} $$
and
$$ R=J(X,\vec{S})-\{z_0\}-L. $$
Then $L$ is obviously path-connected. It suffices to show that $R$ is connected. Let $C$ be the component of $x^*$ in $R$. Given any $z_1\in R$, we show that $z_1\in C$. Suppose $z_1=(u_1, s_1)$ for some $u_1\in I(X,\vec{S})$ and $s_1\in [0,1]$. We consider three subcases.

 Subcase (A1): $u_1=u_0$. Then $s_1>s_0$ and $\{(u_1, s)\,:\, s\in [s_1, 1]\}$ is a path in $R$ from $z_1$ to $x^*$. Therefore $z_1\in C$.
 
 Subcase (A2): $u_1\in \overline{D}_X-\{u_0\}$. In this case we have again that $\{(u_1, s)\,:\, s\in [s_1, 1]\}$ is a path in $R$ from $z_1$ to $x^*$, and hence $z_1\in C$.
 
 Subcase (A3): $u_1\not\in \overline{D}_X$. It follows that $u_1\in X$ and $s_1=0$. Pick any $u_2\in S_k$ and let $z_2=(u_2,0)$. The canonical path from $z_2$ to $x^*$ witnesses that $z_2\in C$. Now $z_1, z_2\in X\times\{0\}$ and $X\times\{0\}$ is connected. Thus $z_1\in C$.
 
This completes the proof of the claim in case (A). We next turn to case (B): $u_0\not\in D_X$. We consider two subcases.

Subcase (B1): $s_0=0$. Thus $z_0\in X\times\{0\}$. Since $J(X,\vec{S})$ is the union of $X\times\{0\}$ with $F_{\overline{D}_X}$, we have that $J(X,\vec{S})-\{z_0\}$ is the union of $(X\times\{0\})-\{z_0\}$ with $F_{\overline{D}_X}-\{z_0\}$. It suffices to see that both are connected. $(X\times\{0\})-\{z_0\}$ is connected since $X$ has no cut-points. For every point $z_1=(u_1,s_1)\in F_{\overline{D}_X}-\{z_0\}$, the canonical path $\{(u_1,s)\,:\, s\in [s_1,1]\}$ from $z_1$ to $x^*$ is in $F_{\overline{D}_X}-\{z_0\}$. Thus $F_{\overline{D}_X}-\{z_0\}$ is in fact path-connected.

Subcase (B2): $s_0>0$. But $s_0\neq 1$. It follows that $u_0\in \overline{D}_X-D_X$. In this case we let 
$$ L=\{(u_0, s)\,:\, s\in [0,s_0)\}\cup (X\times \{0\})$$
and $$R=J(X,\vec{S})-\{z_0\}-L. $$
$L$ is the union of two connected subsets with a common point $(u_0,0)$, hence it is connected. We show that $R$ is also connected by showing that there is a path in $R$ from every element of $R$ to $x^*$. Suppose $z_1=(u_1, s_1)\in R$. Then $u_1\in \overline{D}_X$ and in $J(X,\vec{S})$ there is a canonical path from $z_1$ to $x^*$. If $u_1\neq u_0$ then the canonical path from $z_1$ to $x^*$ is in $R$. If $u_1=u_0$ then it must be that $s_1>s_0$ and again the canonical path from $z_1$ to $x^*$ is in $R$. 

This completes the proof of the claim in case (B). The claim establishes that $x^*$ is the only point $z$ in $J(X,\vec{S})$ so that $J(X,\vec{S})-\{z\}$ has infinitely many components.

We next claim that the following topological properties characterize points $z$ in $E_X\times\{0\}$ within the space $J(X,\vec{S})$:
\begin{enumerate}
\item[($\alpha$)] For any open set $U$ containing $z$, there is an open set $V\subseteq U$ containing $z$ so that $V$ is connected and $V$ contains cut-points of $J(X,\vec{S})$;
\item[($\beta$)]  $z$ is a noncut-point of $J(X,\vec{S})$.
\end{enumerate}
Recall that $E_X$ is the set of all isolated points in $I(X,\vec{S})$. We first verify that points in $E_X\times\{0\}$ satisfy properties ($\alpha$) and ($\beta$). For any $u_0\in E_X$ and $s_0>0$ the set $\{(u_0, s)\,:\, s\in [0,s_0)\}$ is an open neighborhood of $(u_0,0)$ in $J(X,\vec{S})$. It is clear that neighborhoods of this kind form a neighborhood base of $(u_0, 0)$, and each such neighborhood is a homeomorphic copy of $[0,1)$. Properties ($\alpha$) and ($\beta$) clearly follow from these observations.

Next we verify that no other points in $J(X,\vec{S})$ satisfy both ($\alpha$) and ($\beta$). Let $z_0=(u_0,s_0)\in J(X,\vec{S})-(E_X\times\{0\})$. We consider three cases.

Case 1: $u_0\in E_X$ and $s_0\in (0,1]$. In this case $z_0$ is a cut-point of $J(X,\vec{S})$, with $\{(u_0, s)\,:\, s\in [0,s_0)\}$ being one of the components of $J(X,\vec{S})-\{z_0\}$. Thus property ($\beta$) fails for $z_0$.

Case 2: $u_0\in \overline{D}_X-E_X=\overline{E}_X-E_X$. In this case $u_0$ is not isolated in $\overline{D}_X$, and therefore in any open neighborhood $W$ of $u_0$ in $I(X,\vec{S})$ there are infinitely many $u_1\in E_X\cap W$. It follows that in any open neighborhood $U$ of $z_0$ in $J(X,\vec{S})$ there are infinitely many $u_1\in E_X$ so that
$$ \{(u_1, s)\,:\, s\in [0,1]\}\cap U\neq\emptyset. $$
Let $V\subseteq U$ be an open set of $J(X,\vec{S})$ containing $z_0$ but with $x^*\not\in V$. Then $V$ is disconnected. This shows that $z_0$ fails property ($\alpha$).

Case 3: $u_0\not\in \overline{D}_X$. This implies that $u_0\in X$ and $s_0=0$. Moreover, there is an open neighborhood $W\subseteq X$ of $u_0$ so that $W\cap \overline{D}_X=\emptyset$. Let $V=W\times\{0\}$. Then $V$ is an open neighborhood of $z_0$ in $J(X,\vec{S})$. We will show that $z_0$ fails property ($\alpha$) by showing that $V$ does not contain any cut-point of $J(X,\vec{S})$. For this we will show a stronger fact that any $z_1=(u_1,0)$, where $u_1\not\in \overline{D}_X$, is not a cut-point of $J(X,\vec{S})$. Let $C$ be the component of $J(X,\vec{S})-\{z_1\}$ with $x^*\in C$. Let $u_2\in S_k$ and $z_2=(u_2,0)$. Then the canonical path from $z_2$ to $x^*$ witnesses that $z_2\in C$. Since $X$ has no cut-points, $X-\{u_1\}$ is connected. This implies that $X\times\{0\}-\{z_1\}$ is connected. Since $z_2\in X\times\{0\}$, we have that $X\times\{0\}-\{z_1\}\subseteq C$. Any other point of $J(X, \vec{S})$ is of the form $(u,s)$ with $s>0$, where the canonical path from it to $x^*$ witnesses that it belongs to $C$. Thus $C=J(X,\vec{S})-\{z_1\}$. 

We have thus established that $E_X\times\{0\}$ is exactly the set of points in $J(X,\vec{S})$ satisfying ($\alpha$) and ($\beta$). 

The same analysis applies to the $Y$ side. We are now ready to prove the theorem. Suppose $h$ is a homeomorphism from $J(X,\vec{S})$ to $J(Y,\vec{T})$. 

As the base case of the induction, suppose $k=1$. Suppose the length of $\vec{T}$ is $l\geq 1$. In this case $D_X=E_X$ and $S_1=\overline{E}_X\cap X=\overline{E}_X-E_X=\overline{D}_X-E_X$. We also have $T_l=\overline{D}_Y-D_Y\subseteq \overline{D}_Y-E_Y$. From the toplogical characterizations established above, we must have $h(x^*)=y^*$ and $h[E_X\times\{0\}]=E_Y\times\{0\}$. 
Let $T=\overline{E}_Y-E_Y=\overline{D}_Y-E_Y$. We also have $h[S_1\times\{0\}]=T\times\{0\}$. Now consider 
$$\tilde{X}=J(X,S_1)-(S_1\times\{0\}).$$ 
Let $C_X$ be the component of $x^*$ in $\tilde{X}$. Then 
$C_X$ is exactly $F_{\overline{D}_X}-(S_1\times\{0\})$, which implies that $\tilde{X}-C_X= (X\times\{0\})-(S_1\times\{0\})$. On the other hand, let
$$\tilde{Y}=h[\tilde{X}]=J(Y,\vec{T})-(T\times\{0\}).$$ Let $C_Y$ be the component of $y^*$ in $\tilde{Y}$. Then $C_Y$ is exactly $F_{\overline{D}_Y}-(T\times\{0\})$, which implies that $\tilde{Y}-C_Y= (Y\times\{0\})-(T_l\times\{0\})$. Since $h[\tilde{X}]=\tilde{Y}$ and $h[C_X]=C_Y$, we have that
$$ h[(X\times\{0\})-(S_1\times\{0\})]=(Y\times\{0\})-(T_l\times\{0\}). $$
Thus $h[X\times\{0\}]=(Y\times\{0\})\cup (T\times\{0\})$. This implies that $X$ and $Y\cup T$ are homeomorphic. It follows that $l=1$. Otherwise, $D_Y-E_Y\neq\emptyset$ and therefore $T$ contains isolated points of its own, implying that $Y\cup T$ is disconnected, whereas $X$ is connected, a contradiction. When $l=1$, $T=T_1\subseteq Y$, and so $X$ and $Y$ are homeomorphic. In fact, $h[X\times\{0\}]=Y\times\{0\}$.

Finally we consider the inductive step $k>1$. Suppose the length of $\vec{T}$ is $l>1$. From the topological characterizations, we have $h(x^*)=y^*$ and $h[E_X\times\{0\}]=E_Y\times\{0\}$.
Since $E_X\times [0,1]$ is the smallest path-connected subset of $J(X,\vec{S})$ containing $x^*$ and $E_X\times\{0\}$, and ditto for $E_Y\times [0,1]$, we conclude that $h[E_X\times[0,1)]=E_Y\times[0,1)$. 
Now consider $J(X,\vec{S})-(E_X\times[0,1))$, which is clearly homeomorphic  to $J(X, S_1, \dots, S_{k-1})$. Similarly, $J(Y,\vec{T})-(E_Y\times[0,1))$ is homeomorphic to $J(Y, T_1,\dots, T_{l-1})$. The restriction of $h$ witnesses that $J(X, S_1,\dots, S_{k-1})$ is homeomorphic to $J(Y, T_1,\dots, T_{l-1})$. By the inductive hypothesis, we conclude that $k=l$,  $h[X\times\{0\}]=Y\times\{0\}$, and for all $i=1,\dots, k-1$, $h[S_i\times\{0\}]=T_i\times\{0\}$. Now $S_k=\overline{D}_X\cap X=\overline{E}_X\cap X$ and similarly $T_k=\overline{E}_Y\cap Y$, we obtain that $h[S_k\times\{0\}]=T_k\times\{0\}$.

We have thus shown that $(X, \vec{S})\cong (Y, \vec{T})$, and the proof of the theorem is complete.
\end{proof}

Now Theorem~\ref{main} follows immediately from Theorems~\ref{thm:Z} and \ref{thm:J}. It is routine to verify that the reduction maps are Borel. When $X$ is path-connected and the coded subsets are nonempty, the coding space $J(X, \vec{S})$ is path-connected. We have thus obtained the universality of the homeomorphism relation for path-connected continua.

\begin{theorem} The homeomorphic classification problem of all path-connected continua is Borel bireducible with a universal orbit equivalence relation.
\end{theorem}

\end{document}